\documentclass[12pt]{amsart}

\usepackage{amssymb}
\usepackage{amsmath}
\usepackage{amsthm}

\theoremstyle{plain}

\newtheorem{theorem}{Theorem}[section]

\newtheorem{lemma}[theorem]{Lemma}
\newtheorem{proposition}[theorem]{Proposition}

\newtheorem{observation}[theorem]{Observation}
\newtheorem{fact}[theorem]{Fact}

\theoremstyle{definition}

\newtheorem{definition}[theorem]{Definition}

\newtheorem{problem}[theorem]{Problem}

\newtheorem*{question*}{Question}

\newcommand{\R}{\mathbb{R}}
\newcommand{\N}{\mathbb{N}}

\newcommand{\cal}[1]{\mathcal{#1}}

\newcommand{\inte}{\mathrm{int}\:}

\renewcommand{\epsilon}{\varepsilon}
\renewcommand{\phi}{\varphi}
\renewcommand{\tilde}{\widetilde}

\newcommand{\B}{\mathcal{B}}

\newcommand{\lo}{\longrightarrow}

\newcommand{\conv}{\overline{\rm conv}}



\title{A note on point-finite coverings by balls}

\author[C.A.~De~Bernardi]{Carlo Alberto De Bernardi}

\address{Dipartimento di Matematica per le Scienze economiche, finanziarie ed attuariali, Universit\`a Cattolica del Sacro Cuore, 20123 Milano,Italy}

\email{carloalberto.debernardi@unicatt.it}\email{carloalberto.debernardi@gmail.com}

\begin{document}
\begin{abstract}
	
	We provide an elementary proof of a result by
	V.P.~Fonf and C.~Zanco on point-finite coverings of separable Hilbert spaces. Indeed, by using a variation of the famous argument introduced by
	J.~Lindenstrauss
	and R.R.~Phelps \cite{LP} to prove that the unit ball of a reflexive infinite-dimensional Banach space has uncountably many extreme points, we prove the following result.

		\smallskip
	
	{\em 
		Let $X$ be an infinite-dimensional Hilbert space satisfying $\mathrm{dens}(X)<2^{\aleph_0}$, then $X$ does not admit  point-finite coverings by open or closed
		balls, each of positive radius.
	}
	\smallskip

	\noindent In the second part of the paper, we follow the argument introduced by  
	V.P.~Fonf, M.~Levin, and C.~Zanco    in \cite{FonfLevZan14} to prove that the previous result holds also in  infinite-dimensional Banach spaces that are both uniformly rotund and uniformly smooth.

\end{abstract}

\maketitle


\markboth{Carlo Alberto De Bernardi}{A note on point-finite coverings by balls.}

\section{introduction}

A family of subsets of a real normed space $X$ is called a {\em covering} if the union of all its members coincides with $X$.
A covering of $X$ is {\em point-finite} if each point of $X$ is contained in at most finitely many members of the covering.

The problem concerning existence of point-finite coverings of infinite-di\-mensional normed spaces by balls was  considered for the first time in the paper \cite{Klee1} in which V.~Klee asked the following question.

\begin{problem}[{\cite[Question~2.6]{Klee1}}]\label{Klee'sProblem}	Let $\Gamma$ be a cardinal such that $|\Gamma|\geq\aleph_0$, does $\ell_1(\Gamma)$ (respectively $\ell_p(\Gamma)$ for $1<p<\infty$) admit a locally finite (respectively point-finite) covering by closed balls or open balls, each of positive radius?
	\end{problem}

The question above was motivated by the results, contained in the paper itself, implying existence  of a   covering of $\ell_1(\Gamma)$ by  pairwise disjoint closed balls of radius $1$, whenever $\Gamma$ is a suitable uncountable set. 
In \cite{FZ06}, V.P.~Fonf and C.~Zanco generalized Corson's theorem (see Theorem~\ref{corson} below) by proving that {\em if a Banach space $X$ contains an infinite-dimensional closed subspace non-containing $c_0$ then $X$ does not admit any locally finite covering by bounded closed convex bodies}. This completely solved the problem concerning locally finite coverings by balls of $\ell_1(\Gamma)$. 

 More recently, 
V.P.~Fonf and C.~Zanco \cite{FZHilbert} proved that {\em the infinite-di\-mensional separable Hilbert space does not admit point-finite coverings by closed balls of positive radius}. Then V.P.~Fonf, M.~Levin and C.~Zanco \cite{FonfLevZan14} extended the  result above to separable spaces that are both uniformly smooth and uniformly rotund. However, Klee's  problem about point-finite coverings by  balls of $\ell_p(\Gamma)$ spaces ($1<p<\infty$) remained  open in the non-separable case.

 The proof of the result by V.P.~Fonf and C.~Zanco, contained in \cite{FZHilbert},  is based on the following ingredients:
\begin{enumerate}
	\item \cite[Proposition~2.1]{FZHilbert}, a result  excluding existence of certain point-finite families of slices of the unit ball in separable Banach space;
	\item  \cite[Theorem~3.1]{FZHilbert}, a  characterization of separable
	isomorphically polyhedral Banach spaces via existence of point-finite countable coverings by slices of the unit sphere;
	\item the fact that the intersection among two distinct spheres in any Hilbert space lies in some
	hyperplane. Indeed, this is a 3-dimensional  characterization of  inner product spaces  \cite[(15.17)]{Ami};
	\item the fact that  no infinite-dimensional
	dual (and in particular reflexive) Banach space is polyhedral \cite{Lindenstrauss}.
	\end{enumerate} 
  The aim of the present paper is to provide a direct and quite elementary proof of the main result contained in \cite{FZHilbert} and to present an improvement of the result contained in \cite{FonfLevZan14},  concerning point-finite coverings by balls of Banach spaces that are both uniformly smooth and uniformly rotund. Let us start by describing the result contained in Section~\ref{section-point-finite}. Our Proposition~\ref{uniformboundedness} is a  restatement of \cite[Proposition~2.1]{FZHilbert}, the elementary alternative  proof presented in our paper is an immediate application of the uniform boundedness
principle and it works also in the non-separable case.  Theorem~\ref{teo:FollowingLindenstraussPhelps},  excludes existence of certain point-finite families of open or closed slices of the unit ball in reflexive Banach spaces, and it is 
 a variation of the famous argument introduced by
J.~Lindenstrauss
and R.R.~Phelps \cite{LP} to prove that the unit ball of a reflexive infinite-dimensional Banach space has uncountably many extreme points.
Theorem~\ref{teo:FollowingLindenstraussPhelps}, combined with (iii), allows us to obtain the following  slight improvement of \cite[Corollary~3.3]{FZHilbert}.
\smallskip

{\em 
Let $X$ be an infinite-dimensional Hilbert space.
\begin{enumerate}
	\item If  the density character of $X$ satisfies $\mathrm{dens}(X)<2^{\aleph_0}$ then it does not admit  point-finite coverings by open or closed
	balls, each of positive radius.
	\item $X$ does not admit point-finite coverings  by open balls. 
\end{enumerate}
}
\medskip 

Finally,  in Section~\ref{section:appendix}, we observe that, following the argument introduced in \cite{FonfLevZan14}, it is possible to extend this latter result to Banach spaces that are both uniformly smooth and uniformly rotund.  The new ingredients in our proof are  
Lemma~\ref{lemma:interiorconvexhull}, that allows us to deal with  open and closed balls at the same time, and an easy separable reduction argument used in Theorem~\ref{teo:unifrotundsmooth}.
 In particular, our results solve in negative Klee's problem for point-finite coverings by open balls of $\ell_p(\Gamma)$ spaces ($1<p<\infty$).

\section{Point-finite coverings by slices and balls in Hilbert spaces}\label{section-point-finite}

Throughout the paper, we consider only nontrivial real normed spaces. If $X$ is a normed space then $X^*$ is its dual Banach space. We denote by $B_X$, $U_X$, and $S_X$ the closed unit ball, the open unit ball, and the unit sphere of $X$, respectively. We denote by $U(x,\epsilon)$  the open  ball with radius $\epsilon> 0$ and center $x$.
We denote by $B(x,\epsilon)$  the closed  ball with radius $\epsilon\geq 0$ and center $x$; in the case $\epsilon=0$, $B(x,\epsilon)$ is the {\em degenerate ball} containing only the point $x$. In general, by a {\em ball} in $X$ we mean a closed ball of non-negative radius or an open ball of positive radius in $X$.
 For $x,y\in X$, $[x,y]$ denotes the closed segment in $X$ with
endpoints $x$ and $y$, and $(x,y)=[x,y]\setminus\{x,y\}$ is the
corresponding ``open'' segment.
 A set $B\subset X$ will be called a {\em body} if it is closed, convex and has nonempty
interior. A body is called {\em rotund} if its boundary does not contain nontrivial segments.
 Other notation is standard, and various topological notions refer to the norm topology of $X$, if not specified otherwise.

Let $\mathcal{F}$ be a family of nonempty sets in a normed space $X$. By $\bigcup \mathcal F$ we mean the union of all members of $\mathcal F$.
A point $x\in X$ is a {\em regular point} for $\mathcal F$ if it has a neighbourhood that meets at most finitely many members of $\mathcal F$. 
Points that are not regular are called {\em singular}. 
Notice that the set of  singular points is a closed set.

\begin{definition}\label{D:point-locally-finite} The family $\mathcal F$ is called:
	\begin{enumerate}
					\item {\em point-finite}  if each $x\in X$ is contained in at most finitely many members of $\mathcal F$;
		\item {\em locally finite} if each $x\in X$ is a regular point for $\mathcal F$.   
	\end{enumerate}
\end{definition}

\noindent A {\em minimal covering} is a covering whose no proper subfamily is a covering.  A standard application of  Zorn's lemma shows that {\em every point-finite covering contains a minimal subcovering.}

In the sequel, we say that $\mathcal F$ is {\em a family of open or closed balls} of $X$ if  each  element of $\mathcal F$ is an open ball (of positive radius) or a closed ball of non-negative radius (i.e., if not differently stated, we admit that $\mathcal F$ contains also degenerate balls). 

 Let us recall the following famous theorem by H.H.~Corson \cite{Cor61}. 

\begin{theorem}\label{corson}
	Let $\mathcal F$ be a covering of a reflexive infinite-dimensional  Banach space by bounded convex sets. Then $\mathcal F$ is not	locally finite.
\end{theorem}

In what follows, we shall use several times the following fact that  immediately follows  by \cite[Lemma~2.2]{DESOVESTAR}. Let us recall that, if $T$ is a topological vector space, $\mathrm{dens}(T)$ denotes its density character (i.e., the smallest cardinality of a dense subset of $T$).

\begin{fact}\label{density} Let $T$ be a topological space and let $\B$ be a point-finite family of subsets of
$T$. Let us denote $\B':=\{B\in\B;\, \inte B\neq\emptyset\}$, then
$|\B'|\leq\mathrm{dens}(T)$.
\end{fact}

\noindent The following proposition is a restatement of \cite[Proposition~2.1]{FZHilbert}. The elementary alternative proof presented here below is an immediate consequence of the uniform boundedness principle and it works also in the non-separable case.

\begin{proposition}\label{uniformboundedness} Let $X$ be a Banach space.
Let $D\subset X^*$ be an unbounded set. For each $f\in D$, define $S_f=\{x\in X;\, f(x)\geq 1\}$. Then there exist $x\in S_X$ and
an infinite set $N\subset D$ such that $x\in \inte S_f$,
 whenever $f\in N$.
\end{proposition}

\begin{proof}
Suppose on the contrary that, for every $x\in S_X$, the set
$$N_x:=\{f\in D;\, x\in \inte S_f\}$$ is finite. Fix $x\in S_X$ and observe that, since $N_x$ is finite, the set
$x(D)\subset\R$ is upper-bounded. By the Banach-Steinhaus uniform boundedness
principle, we get a contradiction.
\end{proof}

\noindent The following theorem is the core of the results of this section and it is a variation of \cite[Theorem~1.1]{LP}, in which J.~Lindenstrauss
and R.R.~Phelps  proved that the unit ball of a reflexive infinite-dimensional Banach space has uncountably many extreme points.

\begin{theorem}\label{teo:FollowingLindenstraussPhelps} Let $X$ be an infinite-dimensional reflexive Banach space and $\{f_n\}\subset X^*\setminus U_{X^*}$. For each $n\in\N$, let  $S_n$ be one of the following two sets $$\{x\in X;\, f_n(x)\geq
1\},\ \ \ \ \ \{x\in X;\, f_n(x)>1\}.$$ Let us denote $\mathcal S=\{S_n\}_{n\in\N}$ and suppose that  $S_X\subset\bigcup \mathcal S$. Then $\mathcal S$ is not point-finite.
\end{theorem}

\begin{proof}
Suppose on the contrary that, for every $x\in S_X$, the set
$$N_x:=\{n\in\N;\, x\in S_n\}$$ is finite. By Proposition~\ref{uniformboundedness}, we can assume that $\{f_n\}$ is bounded in $X^*$.
  For every $n\in\N$, let $U_n=f_n^{-1}\bigl{((-\infty,1)\bigr)}$ and put $U=\bigcap_n U_n$.
  Then $U$ is a convex set and $0\in\inte U$ (since $\{f_n\}$ is bounded in $X^*$). Moreover, $S_X\cap U=\emptyset$ and hence $U\subset U_X$. 
  
  We claim that $U$ is open.
   To see this, let $x\in U\setminus\{0\}$ and suppose on the contrary that $\sup_n f_n(x)=1$. Then, since $f_n(x)<1$ for each $n\in\N$ and
   since $\|x\|<1$, $N_{x/\|x\|}$ is an infinite set. This contradiction proves our claim.

Now, for every $n\in\N$, put $F_n=\{x\in\overline U;\,
f_n(x)=p_U(x)\}$ (where $p_U$ denotes the Minkowski gauge of the
set $U$) and observe that $F_n$ is closed convex and hence
$w$-closed.

 Fix $x\in \partial U$ and observe that, for each $n\in\N$, $f_n(x)\leq1$; since $x\not\in U$,
 there exists $\overline n\in\N$ such that $f_{\overline n}(x)=1$. Hence $\overline U=\bigcup F_n$.
 Since $\overline U$ is $w$-compact, by the Baire category theorem, we can suppose without any loss of generality that $F_1$ has nonempty interior in $(\overline U,w)$.
 So, there exist $x_0\in F_1\cap U$ and $W$, a neighbourhood of the origin in the $w$-topology, such that
$(x_0+W)\cap\overline U\subset F_1$. Since $X$ is
infinite-dimensional, there exists $y_0\in[x_0+(W\cap\ker
f_1)]\cap\partial U\subset F_1$. Then
$$1=p_U(y_0)=f_1(y_0)=f_1(x_0)=p_U(x_0).$$
A contradiction, since $x_0\in U$.
\end{proof}

\noindent The following observation is an easy consequence of the fact that  the intersection among two distinct spheres in any Hilbert space lies in some
hyperplane (see \cite[(15.17)]{Ami}).

\begin{observation}\label{obs:slice} Let $X$ be a Hilbert space and let $B$ be a closed (open, respectively) ball intersecting the unit sphere $S_X$. Then there exists a
	closed (open, respectively) slice $S$ of $B_X$ such that $S_X\cap B$ coincide with
	$S_X\cap S$.
\end{observation}

\noindent We are now ready to prove the main result of this section.

\begin{theorem}\label{teo:separableHilbert} 
	The following assertions hold true. 
	\begin{enumerate}
		\item Let $\B$ be a covering of a separable infinite-dimensional Hilbert
		space by closed or open balls. Suppose that $\B$ is
		point-finite, then $|\B|= 2^{\aleph_0}$.
		\item If we suppose that $\Gamma$ is an infinite set such that $|\Gamma|<2^{\aleph_0}$,
		$\ell_2(\Gamma)$ does not admit a point-finite covering by open or closed
		balls, each of positive radius.
		\item Let $\B$ be a covering of an  infinite-dimensional Hilbert
		space by open balls. Then $\B$ is
		not point-finite.
	\end{enumerate}
\end{theorem}

\begin{proof}
Let us observe that (ii) follows easily by (i), indeed assume on the contrary that $\B$ is a point-finite cover of $\ell_2(\Gamma)$ by open or closed balls, each of positive radius. Since the
density character of $\ell_2(\Gamma)$ is $|\Gamma|$,  by Fact~\ref{density}, we have $|\B|<2^{\aleph_0}$.
Let us consider $Y=\ell_2\subset\ell_2(\Gamma)$ and  observe
that
$$\B':=\{B\cap Y; B\in\B,\, B\cap Y\neq\emptyset\}$$
is a cover of a separable infinite-dimensional Hilbert space by
open or closed balls such that $|\B'|<2^{\aleph_0}$.
By (i), we get a contradiction.

Similarly, (i) implies (iii). Indeed, if $\B$ is a cover of an infinite-dimensional Hilbert space $X$  by open  balls and we  consider $Y=\ell_2\subset X$, we have
that
$$\B':=\{B\cap Y; B\in\B,\, B\cap Y\neq\emptyset\}$$
is a cover of a separable infinite-dimensional Hilbert space by
open balls. By Fact~\ref{density}, $\B'$ is countable. By (i), $\B'$ (and hence $\B$) is not point-finite. 

It remains to prove (i). Let $\B$ be a point-finite cover of a separable
infinite-dimensional Hilbert space $X$ by open or closed balls.  Since $|X|=2^{\aleph_0}$ and $\B$ is point-finite, we clearly have $|\B|\leq 2^{\aleph_0}$. Now, suppose on the contrary that   $|\B|< 2^{\aleph_0}$. Since the origin of $X$ is
contained in finitely many members of $\B$, if we denote
$$\B':=\{B\in \B;\, 0\not\in B\},$$
 there
exists $R_0>0$ such that, for each $r\geq R_0$, $r S_X$ is
contained in $\bigcup \B'$.

Let us consider the set $A\subset[R_0,\infty)$ defined by
$$A:=\{r\geq R_0;\, \exists B\in\B'\hbox{ such that }B\subset r S_X\};$$
that is, $A$ is the set of all $r\geq R_0$ such that $r S_X$
contains a degenerate ball $B\in\B'$. It is clear that
$|A|\leq|\B'|=|\B|< 2^{\aleph_0}$ and hence there exists $\rho\in [R_0,\infty)\setminus A$. By the
separability of the space, it is clear that the family
$$\B'':=\{B\in \B';\,  B\cap \rho S_X\neq\emptyset\},$$
is countable (indeed, each element in $\B''$ has nonempty
interior). Moreover, $\rho S_X$ is contained in
$\bigcup \B''$. By Observation~\ref{obs:slice}, there
exists a countable point-finite family $\mathcal S$ of closed or open
slices of $\rho B_X$ which covers
$\rho S_X$ and such that $0\not\in\overline S$, whenever $S\in\mathcal S$. By
Theorem~\ref{teo:FollowingLindenstraussPhelps}, we get a contradiction. 
\end{proof}

\section{Point-finite coverings by balls of Banach spaces that are both uniformly rotund and uniformly smooth}\label{section:appendix}

The aim of this section is to show that, following the argument introduced in \cite{FonfLevZan14}, it is possible to extend Theorem~\ref{teo:separableHilbert} to Banach spaces that are both uniformly rotund and uniformly smooth.
 The next two results coincide with \cite[Proposition~2.3]{FonfLevZan14} and {\cite[Fact~2.4]{FonfLevZan14}}, respectively. Observe that, if we use 
  Proposition~\ref{uniformboundedness} instead of \cite[Proposition~2.1]{FZHilbert},
  both the proofs presented in \cite{FonfLevZan14} work also in the non-separable case and even if we consider families of open or closed balls. 

\begin{proposition}
	\label{prop:locally-point-finite}
	
	Let ${\cal B}=\{B_n\}_{n\in\N}$ be a countable family of open or closed balls in a uniformly smooth
	Banach space $X$. Let us denote by $R_n$ the radius of $B_n$ ($n\in\N$) and suppose that $R_n\to \infty$.
	If $\cal B$ is not locally finite, then it is not point-finite.
\end{proposition}

\begin{fact}\label{F2} Let ${\cal B}=\{B_n\}_{n\in\N}$ be a countable collection of open or closed balls in a uniformly rotund
	Banach space $X$. Let us denote by $R_n$ the radius of $B_n$ ($n\in\N$).
	Let  $b>0$ and $x_0\in X$. Suppose that, for each $n\in\N$, $R_n > b$ and $x_0 \notin {\rm int}B_n$. If
	$$F_n = \conv \bigl( B_n \setminus U(x_0,b) \bigr)$$
	and  ${\rm dist}(x_0,F_n) \to 0$
	then $R_n \to \infty$.
\end{fact}

\noindent The next lemma coincides with \cite[Lemma~2.5]{FonfLevZan14}. Observe that in their statement it is not necessary to require that the members of $\cal F$ are closed.

\begin{lemma}
	\label{lemma1}
	Let $X$ be a reflexive Banach space. Let $x_0 \in X$, $a>b>c>0$ and
	$\cal F$ a  collection of convex subsets of $X$ contained in $B(x_0,a)\setminus U(x_0,c)$ such that 
	$\cal F$ covers $B(x_0,a)\setminus U(x_0,b)$. Then $\cal F$ is not locally finite in $X$.
\end{lemma}

\noindent The next lemma coincides with \cite[Lemma~2.6]{FonfLevZan14}. Observe that it holds also in the  case $X'$ is a closed infinite-dimensional subspace. Moreover, in their statement it is not necessary to require that the members of $\cal F$ are closed. Indeed, it is sufficient in its proof to use  Fact~\ref{F2} and Lemma~\ref{lemma1} instead of \cite[Fact~2.4]{FonfLevZan14} and \cite[Lemma~2.5]{FonfLevZan14}, respectively.

\begin{lemma}
	\label{lemma2}
	Let $X$ be both uniformly rotund and uniformly smooth. Consider a closed infinite-dimensional subspace $X' \subset X$ and let $x_0 \in X',\ a >0$. Assume that
	${\cal B}=\{B_n \}_{n=1}^\infty$ is a countable point-finite collection of open or closed
	balls and ${\cal F}=\{F_n \}_{n=1}^\infty$
	is a countable collection of  convex sets such that: 
	$\cal F$ covers $B(x_0,a)\cap X'$,
	$F_n \subset B_n\cap B(x_0,a)$ and
	$x_0 \notin {\rm int} B_n$, whenever $n\in\N$.
	Then there is a point $y \in B(x_0,a) \cap  X', \ y \neq x_0$, that is a singular point for $\cal F$.
\end{lemma}

\begin{lemma}\label{lemma:interiorconvexhull}
Let $A$ be an open convex subset of an infinite-dimensional Banach space  $X$. Let $A_1,\ldots,A_n$ be nonempty convex  sets in $X$ such that, for each $i=1,\ldots,n$ and $x\in\partial A_i$, there exists a hyperplane $\Gamma$  supporting $\overline{A_i}$ at $x$ such that $\Gamma\cap \overline{A_i}=\{x\}$. Define $D=A \setminus(A_1 \cup \dots \cup A_n )$, then 
$D\subset\inte\bigl(\conv (D)\bigr)$.
\end{lemma}

\begin{proof}
Suppose that $x\in D$, let us prove that $x\in\inte\bigl(\conv (D)\bigr)$. If $x\in\inte D$ there is nothing to prove. Suppose that $x\in\partial D$, without any loss of generality, we can suppose that there exists $1\leq m\leq n$ such that:\begin{enumerate}
	\item $x\in\partial A_1\cap\ldots\cap\partial A_m$;
	\item  $x\not\in \overline A_{k}$, whenever $m<k\leq n$.
\end{enumerate}    For each $i=1,\ldots,m$, let   $\Gamma_i$ be a hyperplane supporting $\overline{A_i}$ at $x$ such that $\Gamma_i\cap \overline{A_i}=\{x\}$. Since $X$ is infinite-dimensional, $\Gamma=\Gamma_1\cap\ldots\cap\Gamma_m$ is an infinite-dimensional affine subset of $X$. Since $A$ is open, there exists $\epsilon>0$ such that $B(x,2\epsilon)\subset A\setminus(\bigcup_{m<k\leq n}A_k)$. Let $v_j\in \Gamma\cap B(x,\epsilon)$ ($j=1,2$) be such that $x\in(v_1,v_2)$ and let $0<\delta<\epsilon$ be such that $B(v_j,\delta)\cap \overline{A_i}=\emptyset$ ($j=1,2$, $i=1,\ldots,m$). Then clearly $B(x,\delta)\subset\conv (D)$ and the proof is concluded.   
\end{proof}

Using the previous lemma we obtain the following easy variation of 
\cite[Lemma~2.7]{FonfLevZan14}. For the sake of completeness, we include a proof.

\begin{lemma}
	\label{lemma3}
	Let $X$ be both uniformly rotund and uniformly smooth. Let ${\cal B}=\{ B_n \}_{n\in\N}$ be a  countable point-finite 
	family of open or closed balls in
	$X$.
	Let $Y\subset X$ be a separable infinite-dimensional closed subspace of $X$ and suppose that $B'_n=B_n\cap Y\neq\emptyset$ ($n\in\N$) and that $\{B'_n\}_{n\in\N}$ is a covering of $Y$.
		Put $B^\#_1 =B'_1$ and, for each $n\in\N$, define $$\textstyle B^\#_{n+1}=\begin{cases} \conv ((B'_{n+1} \setminus(B'_1 \cup \dots \cup B'_n )) &\ \text{if}\ B'_{n+1}\ \text{is a closed set in $Y$};\\
	\mathrm{int}_Y\bigl(\conv ((B'_{n+1} \setminus(B'_1 \cup \dots \cup B'_n ))\bigr) &\ \text{if}\ B'_{n+1}\ \text{is an open set in $Y$}.
	\end{cases}$$
	Then ${\cal B}^\#= \{ B_n^\# \}_{n=1}^\infty$ is a point-finite covering of $Y$. Moreover, for every $n\in\N$, we have that
	$B_n^\# \subset B'_n\subset B_n$ 
	and any $x_0 \in \bigcup_{n\in\N}\mathrm{int}_Y\, B'_n $  is a regular point for ${\cal B}^\#$.
\end{lemma}

\begin{proof}
 Observe that, since $X$ is uniformly rotund, for each $n\in\N$, one of the following conditions hold:
	\begin{enumerate}
		\item $\overline{B'_n}^Y$ is a rotund body  in $Y$;
		\item $B'_n$ is a singleton.
	\end{enumerate} 
In any case, for each $n\in\N$ and $x\in\partial_Y B'_n$, there exists a hyperplane $\Gamma$ in $Y$ supporting $\overline{B'_n}^Y$ at $x$ such that $\Gamma\cap \overline{B'_n}^Y=\{x\}$.
	Applying Lemma~\ref{lemma:interiorconvexhull}, we have that 
$B'_{n+1}\setminus (B'_n\cup\ldots\cup B'_1)\subset B_{n+1}^\#$, whenever $n\in\N$. Hence,
${\cal B}^\#$ is a covering of $Y$.
	
		For the latter part we  proceed as in the proof of \cite[Lemma~2.7]{FonfLevZan14}. For $n\in\N$, let us denote by $R_{n}$  the radius of the ball $B_{n}$. Assume on the contrary  that, for some $\tilde n\in\N$,  $x_0 \in \mathrm{int}_Y\, B'_{\tilde n}$
	is a singular point for ${\cal B}^\#$. Then there exists a subsequence of the integers 
	$\{n_i \}_{i\in\N}$ such that, for each $i\in\N$: (a) $n_i > \tilde n$; (b) $ x_0 \notin B'_{n_i}$;
		(c)  for every $j \geq i$, $B(x_0, 1/i)$ 
		intersects the set $B^\#_{n_j}$.

	Note that $B^\# _{n_i} \subset \conv(B'_{n_i} \setminus B'_{\tilde n})$, whenever $i\in\N$. Then
	$B(x_0, 1/i)$ intersects $ \conv(B_{n_j} \setminus B_{\tilde n})$, whenever $i\in\N$ and $j\geq i$. 
	Let $b > 0$ be such
	that $B(x_0,b) \subset B_{\tilde n}$, then it holds
	$$ \conv\bigl(B_{n_j} \setminus B(x_0,b) \bigr)\supset
	\conv(B_{n_j} \setminus B_{\tilde n}).$$ Since, for each $i\in\N$,  $x_0 \notin B_{n_i}$,  from Fact \ref{F2}, we get that 
	$ R_{n_i} \lo \infty $. By Proposition~\ref{prop:locally-point-finite}, this
	contradicts the assumption that $\cal B$ is point-finite.
\end{proof}

\begin{theorem}\label{teo:unifrotundsmooth} 
	Let $X$ be an infinite-dimensional Banach space. Suppose that $X$ is both uniformly rotund and uniformly smooth. Then the following assertions hold true. 
	\begin{enumerate}
		\item 
		If  $\mathrm{dens}(X)<2^{\aleph_0}$ then
		$X$ does not admit a point-finite covering by open or closed
		balls, each of positive radius.
		\item If $\B$ is a cover of $X$ by open balls then $\B$ is
		not point-finite.
	\end{enumerate}
\end{theorem}

\begin{proof}
	(i) Suppose on the contrary that ${\cal B}$ is a point-finite covering of $X$ by open or closed balls, each of positive radius. By Fact~\ref{density}, we have $|\B|<2^{\aleph_0}$.

Let $Y$ be a separable infinite-dimensional closed subspace of $X$. Let us denote
$$\cal B'=\{B\cap Y;\, B\in\B,\ B\cap Y\neq\emptyset  \}.$$
Clearly $\cal B'$ is a point-finite covering of $Y$ and passing to a subcovering we can suppose that $\cal B'$ is a minimal covering of $Y$. 
  If we denote 
$$\cal C'=\{C\in\cal B';\, \mathrm{int}_Y C\neq\emptyset\},\ \ \ \cal D'=\{D\in\cal B';\, |D|=1\},$$
it is clear that $\cal B'=\cal C'\cup \cal D'$ and that $\cal C'$ is countable (since $Y$ is separable). Hence, $\bigcup \cal C'$ is a Borel subset of $Y$. By the fact that $\cal B'$ is minimal we have that $\bigcup \cal D'=Y\setminus \bigcup \cal C'$. Hence, $\bigcup \cal D'$ is a Borel subset of a Polish space such that
$|\bigcup \cal D'|<2^{\aleph_0}$. By
  \cite[Theorem~3.2.7]{Srivastava}, $|\cal D'|=|\bigcup \cal D'|\leq{\aleph_0}$, and hence $\cal B'$ is countable. Let $\cal B'=\{B'_n\}_{n\in\N}$ and suppose that, for each $n\in\N$, $B'_n=B_n\cap Y$ for some $B_n\in\cal B$.
  Now, we proceed as in the proof of  \cite[Theorem~1.5]{FonfLevZan14}.
Consider the covering ${\cal B}^\#$ of $Y$ from Lemma \ref{lemma3} and let
$S\subset Y$ be the set of the points that are singular for ${\cal B}^\#$. By Theorem~\ref{corson}, we have
$S \neq \emptyset $ and, by Lemma~\ref{lemma3}, we have  $S\subset \cup_n  \partial_Y B'_n$.
Since $S$ is closed in $Y$,  by the Baire category theorem, there 
are $m\in\N$, $x_0 \in S$, and  $a>0$  such that
$S \cap B(x_0,a) \subset \partial_Y B'_m$.
Observe that we have two possibilities: $B'_m$ is a singleton or $\overline{B'_m}^Y$ is a  rotund body in $Y$. In any case, there exists a closed hyperplane $Y'$ in $Y$ passing 
through $x_0$ and intersecting $\overline{B'_m}^Y$ only at $x_0$. Then, by applying Lemma~\ref{lemma2} to the families
${\cal F}= \{ B^\#_n \cap B(x_0,a);\, B^\#_n \in {\cal B}^\#\}$ and $\{B_n\}_{n\in\N}$,
 with respect to the subspace $Y'$, we get a contradiction.
 
 The proof of (ii) is similar but easier. Indeed, observe that if $Y$ and $\cal B'$ are defined as above then we clearly have that  $\cal B'$ is countable and then we can proceed as in the previous point.
\end{proof}

\noindent In the non-separable case,  non-existence of coverings by balls satisfying certain condition,   were recently proved in the papers  \cite{DEVETIL, DESOVESTAR}.
  In \cite{DESOVESTAR}, the authors showed that {\em if $X$ is LUR or uniformly smooth then it does not admit star-finite coverings by closed balls, each of positive radius} (we recall that a family of sets is called star-finite if each of its members intersects only finitely many other members of the family). 
The results contained in \cite{DEVETIL} imply that  {\em if $X$ is LUR or Fr\'echet smooth then it does not admit tilings by closed balls}. 
However, the following problem remains open, even in the case $X$ is a Hilbert space.

\begin{problem} Is it possible to generalize (i) in  Theorem~\ref{teo:unifrotundsmooth},  to the case $\mathrm{dens}(X)\geq 2^{\aleph_0}$?
\end{problem}

\section*{Acknowledgment} The research of the author is partially
supported by GNAMPA-INdAM, Project GNAMPA 2020.
The author would like to thank J.~Somaglia, L.~Vesel\'y, and C.~Zanco for many  discussions on the subject and for useful comments and remarks that helped him in preparing this paper.

\end{document}